\documentclass[abstracton]{scrartcl}
\usepackage{amsthm,amssymb,amsmath}
\usepackage{enumitem}
\usepackage{hyperref}
\usepackage{subcaption}
\usepackage{xspace}
\usepackage[utf8]{inputenc}

\usepackage{pgf,tikz,ifthen,calc}
\usetikzlibrary{math,positioning,calc,arrows.meta,decorations.pathmorphing,decorations.markings,shapes}

\newtheorem{theorem}{Theorem}[section]
\newtheorem{lemma}[theorem]{Lemma}

\newtheorem{corollary}[theorem]{Corollary}
\newtheorem{conjecture}[theorem]{Conjecture}
\newtheorem{question}[theorem]{Question}

\newcommand{\C}{\textsc{Cops}\xspace}
\newcommand{\R}{\textsc{Robber}\xspace}
\newcommand{\CnR}{\textsc{Cops and Robber}\xspace}

\theoremstyle{definition}
\newtheorem{assert}[theorem]{Assertion}

\theoremstyle{remark}

\title{On the cop number of toroidal graphs}
\author{Florian Lehner\thanks{The author was supported by the Austrian Science Fund (FWF) Grant no.\ J 3850-N32}}

\begin{document}

\maketitle
\begin{abstract}
    We show that the cop number of toroidal graphs is at most $3$. This resolves a conjecture by Schroeder from 2001 which is implicit in a question by Andreae from~1986.
\end{abstract}

\section{Introduction}

\CnR is a pursuit--evasion game played on a graph between two players. Originally introduced independently by Nowakowski and Winkler \cite{nowakowskiwinkler-copsrobbers}, and Quilliot \cite{quilliot-copsrobbers}, this game and variants thereof have become a quickly growing research area within graph theory. The book \cite{bonatonowakowski-book} provides an extensive introduction to the topic. 

The variant considered in this paper was first studied by Aigner and Fromme \cite{aignerfromme-planar} and can be described as follows. Initially, the first player, called \emph{\C}, places $k$ cops\footnote{Throughout this note, we use \C to refer to the player, and cops to refer to the playing pieces of that player on the graph. An analogous distinction is made between \R and robber.} on the vertices of a graph $G$. Then the second player, called \emph{\R}, places a robber on a vertex. Then the two players take turns. On \C' turn, each cop can either be moved to an adjacent vertex or left at the current position. On \R's turn, the robber can either be moved to an adjacent vertex or left where he is. Both players have perfect information, that is, they know the other player's moves and possible strategies. \C wins the game if at some point one of the cops is at the same vertex as the robber, in this case we say that the robber is caught.

One of the most studied questions concerning this game is whether for some given $k$ there is a winning strategy for \C using $k$ cops. The \emph{cop number} of a graph $G$, denoted by $c(G)$, is the least positive integer $k$ for which \C has a winning strategy. The most famous open problem in this context is Meyniel's conjecture, stating that the cop number of any graph on $n$ vertices is at most $O(\sqrt n)$. If true, this is asymptotically tight since there are graph classes meeting this bound. However, not even an upper bound of the form $O(n^{1-\epsilon})$ is known, see \cite{bairdbonato-meyniel} for an overview.

Bounds for the cop number have also been studied in certain graph classes, with an increased recent interest in graph classes defined by topological invariants, see for example the survey \cite{bonatomohar-topological}. Andreae \cite{andreae-excludedminor} showed that for any fixed graph $H$ there is a constant upper bound on the cop number of connected graphs with no $H$-minor. It follows that there is a constant upper bound on the cop number of connected graphs of genus $g$. In his paper, Andreae also poses the question of finding sharp bounds for the cop number of such graphs in terms of $g$. 

So far, such a bound is only known for $g=0$. Aigner and Fromme \cite{aignerfromme-planar} showed that on any connected planar graph \C has a winning strategy using $3$ cops, and there are planar graphs (such as the dodecahedron) for which $3$ cops are necessary. 
For toroidal graphs $G$, Quilliot \cite{quilliot-genus} proved an upper bound of $c(G) \leq 5$, and Andreae \cite{andreae-excludedminor} asked whether this could be improved to $c(G) \leq 3$. Schroeder \cite{schroeder-genus} improved Quilliots bound to $c(G) \leq 4$, and explicitly stated the conjecture implicit in Andreae's question.

\begin{conjecture}[Andreae, Schroeder]
Let $G$ be a finite toroidal graph, then $c(G) \leq 3$.
\end{conjecture}

In this short note we prove this conjecture. This is done by relating \CnR on a graph $G$ to a similar game with more powerful \C (which we call $T$-\CnR) on a cover of $G$. We note that similar ideas have been used in \cite{clarceetal-nonorientable}, but without increasing the \C' power which is crucial for our proof to work. 

As a corollary to our main result, we are able to make progress on the following conjecture of Schroeder \cite{schroeder-genus}.

\begin{conjecture}[Schroeder]
\label{con:schroeder}
Let $G$ be a finite graph of genus $g$, then $c(G) \leq g+3$.
\end{conjecture}

The best known general bound is $c(G) \leq \frac 43 g+ \frac {10}3$ , proved in \cite{bowleretal-genus}, but so far the conjecture is only known to hold for $g \leq 2$. We give a simpler proof for the case $g=2$, and prove the case $g=3$.

While this confirms Conjecture \ref{con:schroeder} for $g \leq 3$, the bound is only known to be tight for $g=0$. Tightness fails for $g=1$ by our main result, thus raising the following question. 

\begin{question}
Is there any graph with genus $g > 0$ and cop number equal to $g+3$?
\end{question}

Maybe even more fundamentally, we do not know whether the bound in Conjecture~\ref{con:schroeder} is asymptotically tight (Mohar in \cite{mohar} conjectured that it is not), which shows how little is known about the interplay between the genus and the cop number of a graph. In fact, to our best knowledge even the following question is still open.

\begin{question}
What is the smallest $g$ such that there is a graph with genus $g$ and cop number $4$?
\end{question}

\section{Preliminaries}
Throughout this paper, let $G = (V,E)$ be a graph. All graphs considered are simple, undirected, and locally finite (every vertex only has finitely many neighbours). 

An \emph{embedding} of $G$ on a surface $S$ assigns to each vertex $v$ a point $p_v$ on $S$ and to each edge $e = uv$ an arc $a_e$ connecting $p_u$ to $p_v$ such that
\begin{enumerate}
    \item the points $(p_v)_{v \in V}$ are distinct,
    \item the arcs $(a_e)_{e \in E}$ are internally disjoint, and
    \item no point $p_v$ lies in the interior of an arc $a_e$.
\end{enumerate}
Clearly, given a set of points and arcs on a surface with the above properties we can find a graph with this embedding. Call an embedding \emph{accumulation free}, if the set $\{p_v\mid v \in V\}$ has no accumulation points. A graph is called \emph{planar} if it has an embedding in the plane $\mathbb R^2$ and \emph{toroidal} if it has an embedding in the torus $\mathbb T^2 = \mathbb R^2 / \mathbb Z^2$.

Let $d$ denote the usual graph distance on $V$, that is, $d(u,v)$ is the length of a shortest path from $u$ to $v$. For $v\in V$ and $r\in \mathbb N$ we define the \emph{ball around $v$ with radius $r$} by $B_v(r) = \{w \in V \mid d(v,w) \leq r\}$. The ball with radius $1$ around $v$ is called the \emph{closed neighbourhood of $v$} and denoted by $N[v]$. A graph $\hat G = (\hat V,\hat E)$ is a \emph{cover} of $G$, if there is a surjective map $\phi: \hat V \to V$ such that $\phi$ is a bijection from $N[v]$ to $N[\phi(v)]$ for every $v \in \hat V$. The map $\phi$ is called a \emph{covering map}. The \emph{growth function} of $G$ around $v$ is the function $g\colon \mathbb N \to \mathbb N$ defined by $g(n) = |B_v(n)|$. We say that a graph has \emph{polynomial growth}, if the growth function around some (or equivalently any) of its vertices is upper bounded by a polynomial.

The \CnR game on $G$ with $k$ cops is a game played on $G$ between two players, who are called \C and \R respectively. In the beginning of the game, \C picks $(c_0^1,c_0^2,\dots,c_0^k) \in V^k$, then \R picks $r_0 \in V$. In each subsequent turn~$n$, \C picks $c_n^i \in N[c_{n-1}^i]$, then \R picks $r_n \in N[r_{n-1}]$. \C wins the game, if $c_{n+1}^i = r_n$ or $c_n^i=r_n$ for some $n \in \mathbb N$ and some $1 \leq i \leq k$. Note that an optimally playing \R can make sure that the latter option does not happen first, whence we could also insist on $c_{n+1}^i = r_n$ as a winning criterion. The \emph{cop number} $c(G)$ is the least $k$ such that \C has a winning strategy.

Intuitively, we think of the $c_n^i$ and $r_n$ as the position of playing pieces on the graph, \C' playing pieces are thought of as $k$ cops, \R's piece is thought of as a robber. Using this intuition, the winning criterion for \C says that some cop catches the robber by moving to the same vertex. We say that a subgraph $H$ of $G$ is \emph{$i$-guarded at time $n$}, if $r_n \in H$ implies that $c_{n+1}^i = r_n$. Intuitively this means that \C is using the $i$-th cop to make sure that the robber cannot move to $H$ without being caught. Call a subgraph $H$ \emph{guarded}, if it is $i$-guarded for some $i \leq k$.

\section{Main result}

The following result is almost trivial and probably known, but we couldn't find a reference for it in the literature which is why we provide a proof sketch for the convenience of the reader. 

\begin{lemma}
    \label{lem:planarcover}
If $G$ be a finite toroidal graph, then there is an infinite planar cover $\hat G$ of $G$ with polynomial growth. Moreover, $\hat G$ has an accumulation free embedding in $\mathbb R_2$.
\end{lemma}

\begin{figure}
    \centering
    \begin{subfigure}[t]{0.36\textwidth}
    \centering
	\tikzset{->-/.style={decoration={
					markings,
					mark=at position .5 with {\arrow{>[scale=1.5]}}},postaction={decorate}}}
	\tikzset{->>-/.style={decoration={
					markings,
					mark=at position .5 with {\arrow{>[scale=1.5]>[scale=1.5]}}},postaction={decorate}}}
		
	\begin{tikzpicture}[vertex/.style={inner sep=1.5pt,circle,draw,fill}]
		\node (tl) at ($(0:-2)+(90:2)$){};
		\node (bl) at ($(0:-2)+(90:-2)$){};
		\node (tr) at ($(0:2)+(90:2)$){};
		\node (br) at ($(0:2)+(90:-2)$){};
		\path[draw,->>-] (bl.center)--(tl.center);
		\path[draw,->>-] (br.center)--(tr.center);
		\path[draw,->-] (bl.center)--(br.center);
		\path[draw,->-] (tl.center)--(tr.center);
		
	\path [clip] (tr.center)--(br.center)--(bl.center)--(tl.center)--cycle;
	
	{
	\foreach \i in {-1,0,1}
	\foreach \j in {-1,0,1}
	{
	\node[vertex] (c\i\j) at ($(0:4*\i)+(90:4*\j)$){};
	\foreach \a in {0,...,8}
	{
		\node[vertex] (v\a\i\j) at ($({75+\a*40}:1.5)+(c\i\j)$){};
	}
    }
	
	\foreach \i in {-1,0}
	\foreach \j in {-1,0}
	{
	\foreach \a in {0,...,8}
	{
		\pgfmathtruncatemacro{\b}{mod(\a+1,9)} 
		\draw (v\a\i\j)--(v\b\i\j);
	}
	\foreach \a in {0,3,6}
	{
		\draw (v\a\i\j)--(c\i\j);
	}
	
	\pgfmathtruncatemacro{\ii}{\i+1} 
	\pgfmathtruncatemacro{\jj}{\j+1} 
	
	\draw (v1\i\j)--(v5\i\jj);
	\draw (v2\ii\j)--(v7\i\j);
	\draw (v4\ii\jj)--(v8\i\j);
    }
    }
    \end{tikzpicture}
    \caption{}
    \label{subfig:petersen}
\end{subfigure}
\hspace{.5cm}
\begin{subfigure}[t]{0.54\textwidth}
\centering
\begin{tikzpicture}[scale=1.1,vertex/.style={inner sep=1pt,circle,draw,fill}]
	
	\path [clip] (-.45,3.45)--(5.45,3.45)--(5.45,-.45)--(-.45,-.45)--cycle;
 	\foreach \i in {1,...,7}
 	    \draw[dashed] (\i-.5,-1)--(\i-.5,5);
 	\foreach \i in {1,...,3}
 	    \draw[dashed] (-1,\i-.5)--(9,\i-.5);
	
	\foreach \i in {-1,...,8}
	\foreach \j in {-1,...,4}
	{
	\node[vertex] (c\i\j) at ($(0:1*\i)+(90:1*\j)$){};
	\foreach \a in {0,...,8}
	{
		\node[vertex] (v\a\i\j) at ($({75+\a*40}:.35)+(c\i\j)$){};
	}
    }
	
	\foreach \i in {-1,...,7}
	\foreach \j in {-1,...,3}
	{
	\foreach \a in {0,...,8}
	{
		\pgfmathtruncatemacro{\b}{mod(\a+1,9)} 
		\draw (v\a\i\j)--(v\b\i\j);
	}
	\foreach \a in {0,3,6}
	{
		\draw (v\a\i\j)--(c\i\j);
	}
	
	\pgfmathtruncatemacro{\ii}{\i+1} 
	\pgfmathtruncatemacro{\jj}{\j+1} 
	
	\draw (v1\i\j)--(v5\i\jj);
	\draw (v2\ii\j)--(v7\i\j);
	\draw (v4\ii\jj)--(v8\i\j);
    }
    \end{tikzpicture}
    \caption{}
    \label{subfig:cover}
\end{subfigure}
\caption{(\subref{subfig:petersen}) An embedding of the Petersen graph in the torus and (\subref{subfig:cover}) part of the corresponding planar cover constructed in Lemma~\ref{lem:planarcover}. Note that the drawing in each of the dashed squares on the right is exactly the same as the drawing on the left.}
\label{fig:planarcover}
\end{figure}
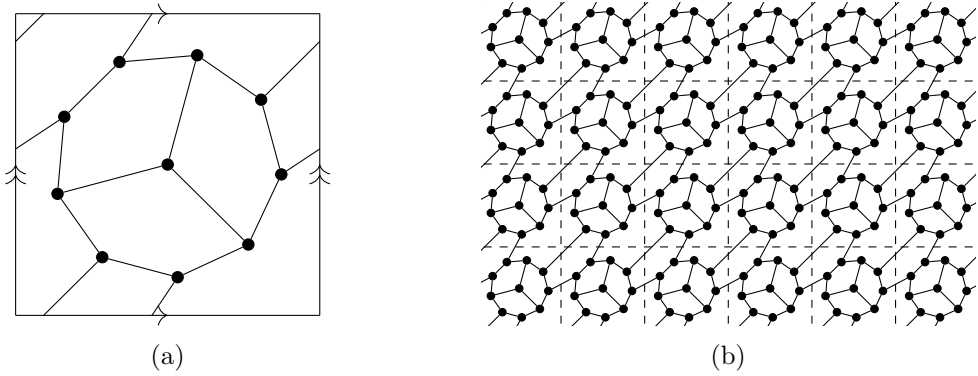

\begin{proof}[Proof sketch]
Let $G$ be a toroidal graph and let $(p_v)_{v \in V},(a_e)_{e \in E}$ be an embedding of $G$ in $\mathbb T^2 = \mathbb R^2/\mathbb Z^2$. Let $\pi \colon \mathbb R^2 \to \mathbb T^2$ be the usual projection map, that is, \ $\pi(x) = x + \mathbb Z^2$. For $v\in V$ define the set $P_v = \pi^{-1}(p_v)$, and for $e \in E$ let $A_e$ be the set of connected components of $ \pi^{-1}(a_e)$. In other words, $P_v$ is the set of all points in $\mathbb R^2$ that project to the embedding $p_v$ of $v$ in $\mathbb T^2$, and $A_e$ is a collection of arcs in $\mathbb R^2$ each of which projects to the embedding $a_e$ of $e$ in $\mathbb T^2$, see Figure~\ref{fig:planarcover} for an example. It is readily verified that the set of points $P = \bigcup_{v \in V} P_v$ together with the set of arcs $A = \bigcup_{e \in E} A_e$ defines an accumulation free embedding of a graph $\hat G = (\hat V , \hat E)$ in the plane and that the projection $\pi$ gives rise to a covering map by mapping $\hat v$ to $v$ if $\pi(p_{\hat v})=p_v$. 

To show polynomial growth, note that in the embedding of $\hat G$ defined above, exactly $|V|$ vertices embed into any translate of $[0,1)^2$. Since any two arcs in $A_e$ can be mapped into each other by a translation, there is an absolute upper bound $R$ on the Euclidian distance of the embeddings of two neighbours in $\hat G$. Consequently, the embeddings of all vertices in $B_v(r)$ are contained in some translate of $[-r R,r R + \epsilon)^2$ and thus $B_v(r)$ contains at most $(2 r R + \epsilon)^2 \cdot |V|$ vertices.
\end{proof}

Given an equivalence relation $T$ on $V$ we can define the following variant of \CnR, which we call $T$-\CnR. The rules are the same as in the original game, except \C is able to `teleport cops to an equivalent position' before moving them. More formally, she can pick $\tilde c_{n-1}^i \, T \, c_{n-1}^i$ and choose $c_n^i \in N[\tilde c_{n-1}^i]$. The \emph{$T$-cop number} $c^T(G)$ is the least $k$ such that \C has a winning strategy using $k$ cops in $T$-\CnR. Note that we do not allow \R\ to teleport---this is essential for Lemma~\ref{lem:guardpath} which otherwise would have to be replaced by an even more technical statement.

For the remainder of this section, we will use the following assertion. Note that $G$ is required to be finite while $\hat G$ can be arbitrary.
\begin{assert}
\label{ass:general}
Let $\hat G = (\hat V, \hat E)$ be a cover of a finite graph $G = (V,E)$ with covering map $\phi$, and let $T$ be the equivalence relation defined by $v \, T \, w$ if and only if $\phi(v) = \phi(w)$.
\end{assert}

The following lemma connecting \CnR on $G$ to $T$-\CnR on $\hat G$ is very similar to \cite[Lemma 1]{clarceetal-nonorientable}.

\begin{lemma}
\label{lem:copnumber-cover}
Under Assertion \ref{ass:general} we have  $c(G) \leq c^T(\hat G)$. 
\end{lemma}

\begin{proof}
Assume that \C has a winning strategy for $T$-\CnR with $k$ cops on $\hat G$. We define a strategy of \C on $G$ by projecting such a winning strategy onto $G$. 

More precisely, given the initial position $r_0$ chosen by \R on $G$, we pick $\hat r_0$ arbitrarily with $\phi(\hat r_0) = r_0$. For $n \geq 1$, let $r_n$ be the position of \R on $G$ at time $n$, and inductively pick $\hat r_n \in N[\hat r_{n-1}]$ such that $\phi(\hat r_n) = r_n$. Note that $\hat r_n$ is unique since $\phi$ is a covering map. The strategy for \C on $G$ is given by $(\phi(\hat c_n^1),\dots,\phi(\hat c_n^k))$ where $(\hat c_n^1,\dots,\hat c_n^k)$ is the position of \C on $\hat G$ with respect to the winning strategy played against $\hat r_n$. Note that this is a valid strategy because $\phi$ is a covering map.

Since the strategy of \C on $\hat G$ is winning, there is some $n$ and $i$ such that $\hat c_n^i = \hat r_{n-1}$. This clearly implies $c_n^i = \phi (\hat c_n^i) = \phi(\hat r_{n-1}) = r_n$, so \C wins the game on $G$ in the same move or earlier.
\end{proof}

A weaker version of the next lemma can be found in \cite{aignerfromme-planar}. The advantage of our version is that we can use the additional power of \C in $T$-\CnR to obtain a bound the distance between $u$ and $r_j$ until the path $P$ is guarded. This will be essential in the proof of our main result.

\begin{lemma}
\label{lem:guardpath}
Assume Assertion \ref{ass:general}, let $u,v \in \hat V$, and let let $P$ be a shortest $u$-$v$-path. Let $r_n$, $c_n^i$ be positions in $T$-\CnR on $\hat G$ with $k$ cops at time $n$, and let $i_0 \in \{1,\dots,k\}$. Then there is a strategy for \C such that for some $m > n$ the following hold:
\begin{enumerate}
    \item $d(u,r_j) \leq d(u,r_n) + |V|$ for $n \leq j \leq m$,
    \item $P$ is $i_0$-guarded at all times $j \geq m$.
\end{enumerate}
Furthermore this strategy does not depend on how $c_j^i$ evolve for $i \neq i_0$ (the value of $m$, however, depends on \R's strategy which in turn may depend on all $c_j^i$).
\end{lemma}

\begin{proof}
Without loss of generality take $i_0 = 1$ and $n = 0$. We give a strategy with the desired properties. 

Let $D$ be the length of $P$ and let $x$ be the unique vertex on $P$ satisfying $d(u,x) = \min(D, d(u,r_0) + |V|)$---uniqueness follows from the fact that $P$ is a shortest path. Since $\phi$ is a covering map, it can be used to lift any path from $\phi(x)$ to $\phi(c_0^1)$ in $G$ to a path from $x$ to some $\tilde c_0^1 \, T \, c_0^1$ in $\hat G$. The distance between $\phi(x)$ and $\phi(c_0^1)$ in $G$ is at most $|V|$, thus there is some $\tilde c_0^1 \, T \, c_0^1$ in $\hat V$ such that $d(x,\tilde c_0^1) \leq |V|$. 

The strategy is as follows, see Figure~\ref{fig:guardpath}.
By teleporting to $\tilde c_0^1$ and then choosing $c_{j+1}^1$ as close as possible to $x$, \C ensures that $c_{j}^1 = x$ for some $j \leq |V|$, in particular, she can make sure that $c_{|V|}^1 = x$. 

For $j > |V|$ we proceed as follows. Let $r'_j$ be the unique vertex on $P$ at distance $\min (d(u,r_j) , D)$ from $u$. If $c_j^1 = r_j'$, then \C chooses $c_{j+1}^1 = c_j^1$, otherwise $c_{j+1}^1$ is the neighbour of $c_j^1$ on $P$ which lies closer to $r_j'$ than $c_j^1$. Independence of this strategy from $c_j^i$ for $i \neq 1$ is obvious.

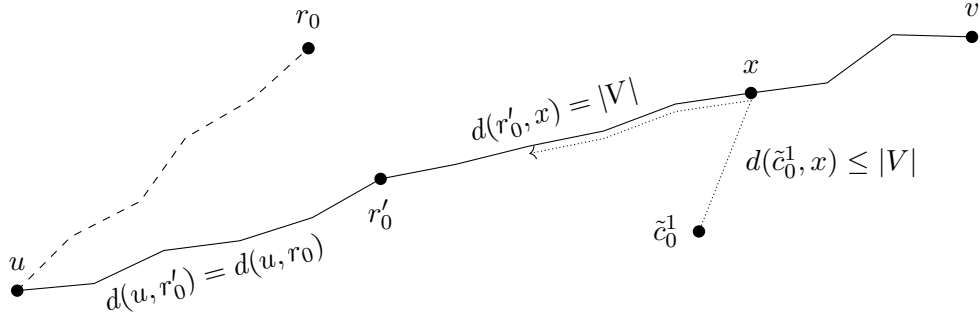
\begin{figure}
    \centering
    \begin{tikzpicture}[vertex/.style={inner sep=1.5pt,circle,draw,fill}]
	\pgfdeclarelayer{lines}
	\pgfdeclarelayer{dots}
	\pgfdeclarelayer{labels}
	\pgfsetlayers{lines,dots,labels}
	\pgfmathsetseed{1234567}
	
	\begin{pgfonlayer}{dots}
	    \node[vertex] (u) at (0,0){};
	    \node[coordinate] (v0) at (u){};
	    \node[vertex] (v) at (15:13){};
	    \node[coordinate] (v13) at (v){};
	    \foreach \i in {1,...,12}
	         \node[coordinate] (v\i) at ($ (15:\i) + rand*(105:.3)$){};
	    \foreach \i in {1,...,4}
	         \node[coordinate] (w\i) at ($ (40:\i) + rand*(130:.3)$){};
	    \node[vertex] (x) at (v10){};
	    \node[vertex] (r0dash) at (v5){};
	    \node[vertex] (r0) at (40:5){};
	    \node[vertex] (c0) at (5:9){};
	\end{pgfonlayer}
	
	\begin{pgfonlayer}{lines}
	    \draw {(v0) \foreach \i in {1,...,13} {--(v\i)}};
	    \draw[dashed] {(v0) \foreach \i in {1,...,4} {--(w\i)}--(r0)};
	    
	    \draw[densely dotted,->]{(c0) -- ($ (x) + (0,-.1) $)\foreach \i in {9,8,7} {--($ (v\i) + (0,-.1) $)}};
	\end{pgfonlayer}

	\begin{pgfonlayer}{labels}
	    \node [label={$u$}] at (u) {};
	    \node [label={$v$}] at (v) {};
	    \node [label={$x$}] at (x) {};
	    \node [label={$r_0$}] at (r0) {};
	    \node [label={below:$r_0'$}] at (r0dash) {};
	    \node [label={left:$\tilde c_0^1$}] at (c0) {};
	    \node [label={[rotate=15]below:$d(u,r_0') = d(u,r_0) $}] at (15:2.6) {};
	    \node [label={[rotate=15]above:$d(r_0',x) = |V|$}] at (15:7.4) {};
	    \node [label={right:$d(\tilde c_0^1,x) \leq  |V|$}] at ($1/2*(c0) + 1/2*(x) $) {};
	\end{pgfonlayer}
	
    \end{tikzpicture}
    \caption{Situation in the proof of Lemma~\ref{lem:guardpath}. The solid line is $P$, the dotted line indicates, how $c_j^1$ develops after teleportation to $\tilde c_0^1$.}
    \label{fig:guardpath}
\end{figure}

Note that $d(u,r_j') \leq \min (D , d(u,r_0) + |V|) = d(u,x)$ for $j \leq |V|$, and thus $d(u,r_{|V|}') \leq d(u,c_{|V|}^1)$. Since $r_{j+1}'$ is contained in the closed neighbourhood of $r_j'$ in $P$, there must be some $m \geq |V|$ such that for $|V| \leq j <m$ we get that $c_{j+1}^1$ is the neighbour of $c_j^1$ which lies closer to $u$ and for $j\geq m$ we get $c_{j+1}^1 = r_j$. Note that if $d(u,r_{|V|}) \geq D$, then $r_{|V|'} = c_{|V|}^1 = v$ whence $m = |V|$ and thus $d(u,r_j) \leq d(u,r_0) + |V|$ for every $j \leq m$. Otherwise, clearly $d(u, r_j) = d(u, r_j') \leq d(u,c_j^1) \leq d(u,r_0) + |V|$ for $j \leq m$, thus proving the first claimed property. The second property follows from the fact that if $r_j \in P$ for $j \geq m$, then $r_j = r_j' = c_{j+1}^1$.
\end{proof}

The next lemma is already implicit in \cite{aignerfromme-planar}. We provide a proof for the sake of completeness---essentially this is the same as the proof of  \cite[Theorem~6]{aignerfromme-planar}, starting in situation (b), described on \cite[page~9]{aignerfromme-planar} which roughly corresponds to condition \ref{itm:twopaths} below. 

The basic idea of the strategy is that \C always has the robber surrounded by two cops. More formally, \C will ensure that the following condition is satisfied with respect to some fixed embedding.
\begin{enumerate}[label=(\textasteriskcentered)]
    \item \label{itm:twopaths} There are paths $P$ and $Q$ (one of which may be empty), a finite component $R$ of $G\setminus(P\cup Q)$, and $m \in \mathbb N$ such that
    \begin{enumerate}
        \item $P$ and $Q$ embed on the boundary of the outer face of the graph induced by $P\cup Q \cup R$,
        \item $r_m \in R$, and
        \item \C has a strategy such that $P$ is $i$-guarded and $Q$ is $i'$-guarded with $i \neq i'$ for every $j \geq m$.
    \end{enumerate}
\end{enumerate}
Note that we allow the paths to be empty, that is, we consider the empty graph as a path. The reason for this is that it reduces the casework involved in the proof. It is also worth noting that if $G$ is a finite, connected planar graph, then \ref{itm:twopaths} can easily be satisfied for a path $P$ consisting of one vertex and an empty path $Q$. In particular, the lemma below implies that any such graph satisfies $c(G) \leq 3$.

\begin{lemma}
\label{lem:planarwin}
Let $G$ be a (potentially infinite) connected planar graph with some fixed planar embedding. If \ref{itm:twopaths} holds in \CnR with $3$ cops, then \C has a winning strategy.
\end{lemma}

\begin{proof}
We proceed inductively. After each iteration, unless \C has won the game, we end up in  situation \ref{itm:twopaths}. The values $|R|$ and $|P|+|Q|+|R|$ never increase in an iteration. Moreover at least one of the two values decreases unless one of the paths is empty before the iteration (in which case both paths are non-empty afterwards). Since $|P|$, $|Q|$, and $|R|$ are non-negative integers this implies that there is some finite upper bound on the number of iterations which means that \C must eventually win the game.

We now turn to the iterative definition of the strategy. If $Q$ has an endpoint $q$ that either lies on $P$ or is not incident to $R$, then replace $Q$ by $Q' = Q \setminus\{q\}$. Since $Q$ is $i'$-guarded at time $j \geq m$, so is $Q'$. In particular, $P$, $Q'$, and $R$ satisfy the conditions of \ref{itm:twopaths}. Note that we left $R$ unchanged and decreased the value of $|P|+|Q|+|R|$.

Since $P$ and $Q$ both lie on the boundary of the outer face, they can only intersect if an endpoint of one of them lies on the other. Thus by iterating the above argument (and possibly exchanging the roles of $P$ and $Q$) we can assume that either $Q$ is empty, or $P$ and $Q$ are both non-empty and disjoint, and all of their endpoints are incident to $R$.

First assume that $Q$ is the empty path, and without loss of generality assume that $P$ is $1$-guarded. If $P$ contains all vertices on the boundary of the outer face of the subgraph induced by $P\cup R$, then let $q$ be an endpoint of $P$, and let $P' = P\setminus\{q\}$. Otherwise, let $q$ be a vertex that is incident to the outer face but does not lie on $P$, and let $P' = P$. Let $Q'$ be the path consisting only of the vertex $q$. Note that in both cases $P'$ and $Q'$ are non-empty (as claimed) because the boundary of the outer face contains at least $3$ vertices.

Since $P$ is $1$-guarded at time $j \geq m$, so is $P'$. Moreover, \C can ensure that for some $m' \geq m$ and all $j \geq m'$ we have that $c_j^2 = q$. If $r_j \in P$ for some $m \leq j \leq m'$ or $r_{m'} = q$, then \C has won the game. Otherwise, $r_{m'}$ is contained in some component $R'$ of $R \setminus\{q\}$ and $P'$, $Q'$, and $R'$ satisfy \ref{itm:twopaths}. Note that $|R'| \leq |R|$ since $R' \subseteq R$ and $|P'|+|Q'|+|R'| = |P| + |Q| + |R|$ since $q$ was contained in exactly one of $P$ and $R$.

Finally, assume that $P$ and $Q$ are non-empty and disjoint, and their endpoints have neighbours in $R$. Orient $P$ and $Q$ in the same direction along the boundary of the outer face and let $p$ and $q$ be their first vertices, respectively. Let $H'$ be the subgraph of $G$ induced by $p$, $q$, and all vertices in $R$. If $pq$ is an edge, then let $H = H'\setminus\{pq\}$, otherwise let $H = H'$. Let $S$ be a shortest $p$-$q$-path in $H$---note that such a path exists because $R$ is connected, and both $p$ and $q$ are incident to $R$. Figure \ref{fig:planarstrategy} illustrates the resulting situation.

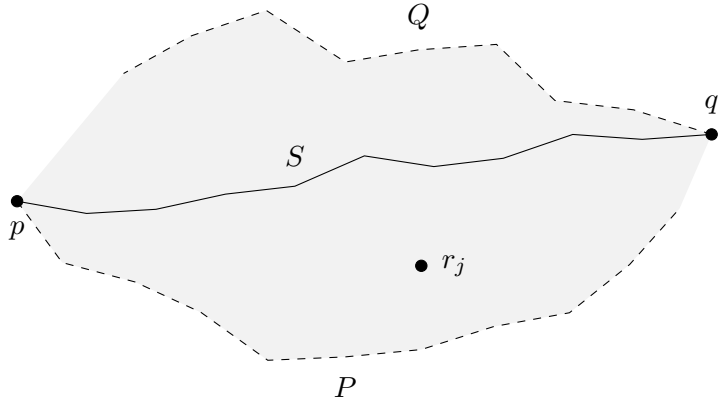
\begin{figure}
    \centering
    \begin{tikzpicture}[vertex/.style={inner sep=1.5pt,circle,draw,fill}]
	\pgfdeclarelayer{lines}
	\pgfdeclarelayer{dots}
	\pgfdeclarelayer{labels}
	\pgfdeclarelayer{bg}
	\pgfsetlayers{bg,lines,dots,labels}
	\pgfmathsetseed{1234567}
	
	\begin{pgfonlayer}{dots}
	    
	    \foreach \i in {0,...,10}
	         \node[coordinate] (p\i) at ($ (230+8*\i:7) + rand*(230+8*\i:.5)+(0,5)$){};
	    \node[vertex] (p) at (p0){};
	    \foreach \i in {0,...,8}
	         \node[coordinate] (q\i) at ($ (50+8*\i:7) + rand*(50+8*\i:.5)+(0,-5)$){};
	    \node[vertex] (q) at (q0){};
	    \node[coordinate] (s0) at (q0){};
	    \node[coordinate] (s10) at (p0){};
	    \foreach \i in {1,...,9}
	         \node[coordinate] (s\i) at ($ \i/10 *(p0) + (q0) - \i/10*(q0) +rand*(0,.3)$){};
	   \node[vertex] (rn) at (1,-1){};
	\end{pgfonlayer}
	
	\begin{pgfonlayer}{lines}
	    \draw[dashed] {(p0) \foreach \i in {1,...,10} {--(p\i)}};
	    \draw[dashed] {(q0) \foreach \i in {1,...,8} {--(q\i)}};
	    \draw {(s0) \foreach \i in {1,...,10} {--(s\i)}};
	    
	\end{pgfonlayer}
	
	\begin{pgfonlayer}{bg}
	\path[fill= gray!10]{(p0) \foreach \i in {1,...,10} {--(p\i)}\foreach \i in {0,...,8} {--(q\i)}--cycle};
	\end{pgfonlayer}
	
	\begin{pgfonlayer}{labels}
	    \node [label={below:$p$}] at (p) {};
	    \node [label={above:$q$}] at (q) {};
	    \node [label={below:$P$}] at (p5) {};
	    \node [label={above:$Q$}] at (q4) {};
	    \node [label={above:$S$}] at (s6) {};
	    \node [label={right:$r_j$}] at (rn) {};
	\end{pgfonlayer}
	
    \end{tikzpicture}
    
    \caption{Situation in the proof of Lemma~\ref{lem:planarwin}: $R$ is embedded in the shaded region of $\mathbb R^2$, all neighbours of $R$ lie in $R$, $P$, or $Q$. Note that any path in $R$ from $r_j$ to $Q$ must contain a vertex of $P$ or $S$.}
    \label{fig:planarstrategy}
\end{figure}

Assume without loss of generality that $P$ is $1$-guarded and $Q$ is $2$-guarded. If $r_j \in P \cup Q$ for some $j \geq m$, then \C wins the game. Otherwise $r_j \in R \subseteq H$ for all $j \geq m$, and by Lemma~\ref{lem:guardpath} (where $T$ is the equality relation on the vertex set of $H$) there is a strategy for \C such that for some $m' \geq m$ the paths $P$, $Q$, and $S$ are $1$-, $2$-, and $3$-guarded at all times $j \geq m'$, respectively.

By the above discussion, we already know that $r_{m'} \in R$. If $r_{m'} \in S$, then \C wins the game since $S$ is $3$-guarded. If not, then let $R'$ be the component of $R \setminus S$ containing $r_{m'}$. Note that due to the choice of $S$, either $R'$ has no neighbours in $P \setminus \{p\}$, or it has no neighbours in $Q \setminus \{q\}$. Without loss of generality assume the latter. Observe that $P \setminus \{p\}$, $S$, and $R'$ satisfy the conditions of \ref{itm:twopaths}. Since $S$ contains at least one inner point in $R$, we have $|R'| < |R|$. Furthermore, $P\setminus \{p\}$, $S$, and $R'$ are disjoint subsets of $P\cup Q \cup R$, whence $|P\setminus \{p\}| + |S| + |R'| \leq |P\cup Q \cup R| \leq |P|+|Q|+|R|$, thus completing the proof.
\end{proof}

\begin{theorem}
\label{thm:main}
If $G=(V,E)$ is a finite toroidal graph, then $c(G) \leq 3$
\end{theorem}

\begin{proof}
In Assertion \ref{ass:general}, let $\hat G$ be a cover of $G$ embedded in the plane as in Lemma \ref{lem:planarcover}. By Lemma \ref{lem:copnumber-cover} it is enough to show that $c^T(\hat G) \leq 3$.

Assume that \C and \R have picked initial positions $(c_0^i)_{i \leq 3}$ and $r_0$ respectively. Choose $D$ large enough that 
\[\frac{D}{|V|} > \log(|B_{r_0}(D)|),\] 
where $\log$ denotes the base $2$ logarithm and $B_{r_0}(D)$ is the ball in $\hat G$. This is possible because $\hat G$ has polynomial growth and $V$ is finite. 

Let $T$ be a shortest path tree of $B_{r_0} (D)$ in $\hat G$ rooted at $r_0$, that is, the unique path in $T$ connecting $r_0$ to $v$ is a shortest $r_0$-$v$-path in $\hat G$ for every $v \in B_{r_0}(D)$.
Let $(v_i)_{1 \leq i \leq l}$ be the vertices at distance $D$ from $r_0$ which are connected by an edge to an infinite component of $G \setminus B_{r_0}(D)$, enumerated in the cyclic order given by the embedding of $T$ in $\mathbb R^2$. For convenience we define $v_0 = v_l$. Note that trivially $l < |B_{r_0}(D)|$. For $0 \leq i \leq l$, denote by $P_i$ the path from $r_0$ to $v_i$ in $T$. 

For $a<b$, denote by $[a,b]=\{v_i \mid a<i<b\}$ and $[b,a]=\{v_i \mid i<a \text{ or }i>b\}$. Let $a < b$ and let $H$ be the graph obtained from $B_{r_0}(D)$ removing the union of $P_a$ and $P_b$. We claim that in there is no path connecting $[a,b]$ to $[b,a]$ in $H$. 

For the proof of this claim, we first note that the embedding of any finite cycle $C$ in $G$ is a simple closed curve $K$ whence $\mathbb R^2\setminus K$ has two connected components one of which is bounded. We say that $C$ \emph{surrounds} a vertex $v$, if $v$ embeds in this bounded connected component. If there is a cycle $C$ in $B_{r_0}(D)$ surrounding $v_a$, then the infinite component of $G \setminus B_{r_0}(D)$ connected to $v_a$ must completely embed in the bounded connected component. The closure of a bounded subset of $\mathbb R^2$ is compact, so there is an accumulation point of vertices, contradicting the fact that the embedding from Lemma~\ref{lem:planarcover} is accumulation free. Thus there cannot be a cycle in $B_{r_0}(D)$ surrounding $v_a$, and the same clearly holds for $v_b$.

Assume now for a contradiction that there is a path $Q$ in $H$ connecting $v_i \in [a,b]$ to $v_j \in [b,a]$. Let $P_{ij}$ be the path from $v_i$ to $v_j$ in $T$, and let $P_{a}'$ and $P_{b}'$ be the paths in $T$ connecting $v_a$ and $v_b$ to $P_{ij}$ respectively. Analogously define $P_{ab}$, $P_i'$, and $P_j'$. Since $v_i \in [a,b]$ and $v_j \in [b,a]$, we either have $a < i < b < j$, or $j < a  < i < b$. The cyclic order of $v_a$, $v_i$, $v_b$, and $v_j$ in $T$ is the same in both cases, and it implies that $P_a'$ and $P_b'$ attach to different sides of $P_{ij}$ in the embedding. Let $x$ be the last vertex on $P_i'$ that lies on $Q$, and let $y$ be the last vertex on $P_j'$ that lies on $Q$. Figure~\ref{fig:disconnected} illustrates the resulting situation.

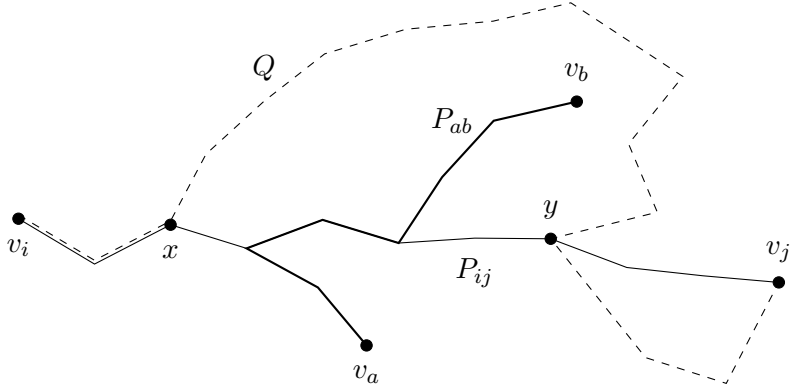
\begin{figure}
    \centering
    \begin{tikzpicture}[vertex/.style={inner sep=1.5pt,circle,draw,fill}]
	\pgfdeclarelayer{lines}
	\pgfdeclarelayer{dots}
	\pgfdeclarelayer{labels}
	\pgfdeclarelayer{bg}
	\pgfsetlayers{bg,lines,dots,labels}
	\pgfmathsetseed{12345678}
    
    \begin{pgfonlayer}{dots}
	    
	    \foreach \i in {0,...,10}
	         \node[coordinate] (p\i) at ($ (\i,0) + rand*(0,0.5)$){};
	   \node[vertex] (vi) at (p0) {};
	   \node[vertex] (vj) at (p10) {};
	   \node[vertex] (x) at (p2) {};
	   \node[vertex] (y) at (p7) {};
	   
	   \node[coordinate] (a0) at (p3){};
	    \foreach \i in {1,2}
	         \node[coordinate] (a\i) at ($ (a0) + (-50:\i) + rand*(40:0.5)$){};
	    \node[vertex] (va) at (a2){};

	   \node[coordinate] (b0) at (p5){};
	    \foreach \i in {1,2,3}
	         \node[coordinate] (b\i) at ($ (b0) + (40:\i) + rand*(-50:0.5)$){};
	    \node[vertex] (vb) at (b3){};

	    \foreach \i in {0,1,2}
	         \node[coordinate] (q\i) at ($ (p\i) + (0,0.05)$){};
	    \foreach \i in {3,...,9}
	         \node[coordinate] (q\i) at ($ (q2) + (90-8*\i:\i-2) + rand*(-8*\i:0.3)$){};
	    \foreach \i in {10,11}
	         \node[coordinate] (q\i) at ($ (q8) + (20-8*\i:\i-8) + rand*(-70-8*\i:0.3)$){};
	   \node[coordinate] (q12) at (p7){};
	    \foreach \i in {13,14}
	         \node[coordinate] (q\i) at ($ (q12) + (8*\i-150:\i-11) + rand*(8*\i-60:0.3)$){};
	   \node[coordinate] (q15) at (p10){};
	         
	\end{pgfonlayer}
	
	\begin{pgfonlayer}{lines}
	    \draw {(p0) \foreach \i in {1,...,10} {--(p\i)}};
	    \draw[thick] {(a2) \foreach \i in {1,0} {--(a\i)}--(p4) \foreach \i in {0,...,3} {--(b\i)}};
	    \draw[dashed] {(q0) \foreach \i in {1,...,15} {--(q\i)}};
	\end{pgfonlayer}
	
	\begin{pgfonlayer}{bg}
	\end{pgfonlayer}
	
	\begin{pgfonlayer}{labels}
	    \node [label={below:$v_i$}] at (vi) {};
	    \node [label={above:$v_j$}] at (vj) {};
	    \node [label={below:$v_a$}] at (va) {};
	    \node [label={above:$v_b$}] at (vb) {};
	    \node [label={below:$x$}] at (x) {};
	    \node [label={above:$y$}] at (y) {};
	    \node [label={above:$Q$}] at (q4) {};
	    \node [label={left:$P_{ab}$}] at (b2) {};
	    \node [label={below:$P_{ij}$}] at (p6) {};
	\end{pgfonlayer}
    \end{tikzpicture}
    \caption{Embedding of the paths $Q$, $P_{ij}$ and $P_{ab}$ in the proof of Theorem \ref{thm:main}. Note that $P_{ij}$ and $P_{ab}$ are both contained in the tree $T$ and thus must intersect in a common subpath.}
    \label{fig:disconnected}
\end{figure}

Let $P_{xy}$ and $Q_{xy}$ be the subpaths of $P_{ij}$ and $Q$ connecting $x$ and $y$, respectively. Note that $P_{xy} \subseteq P_i \cup P_j \cup P_{ab}$. Further note that $P_{ab} \subseteq P_a \cup P_b$, so $Q$ is disjoint from $P_{ab}$. Moreover, by definition $Q \cap P_i \cap P_{xy} = \{x\}$ and  $Q \cap P_j \cap P_{xy} = \{y\}$, and thus $P_{xy}$ and $Q_{xy}$ only meet in their endpoints. Consequently $P_{xy} \cup Q_{xy}$ is a cycle. The vertices $x$ and $y$ are contained in $Q$ and thus not in $P_{ab}$, whence $P_a'$ and $P_b'$ attach to different sides of $P_{xy}$ in the embedding. Since neither of them crosses $P_{xy}$ or $Q_{xy}$ we conclude that the cycle $P_{xy} \cup Q_{xy}$ surrounds either $v_a$ or $v_b$, which yields the desired contradiction and thus proves our claim.

From now on we say that $v \in B_{r_0}(D)$ lies \emph{between $a$ and $b$} if it lies in the same component of $H$ as some element of $[a,b]$. By the above claim, no vertex is between $a$ and $b$ and between $b$ and $a$ simultaneously, but there may be vertices in $B_{r_0}(D)$ which are neither between $a$ and $b$ nor between $b$ and $a$. Clearly, any such vertex is contained in a finite component of $\hat G - (P_a \cup P_b)$, since any path connecting it to $\hat G - B_{r_0}(D)$ must cross either $P_a$ or $P_b$.

We say that \R is \emph{trapped between $a$ and $b$} at time $j$, if $r_j$ lies between $a$ and $b$, and $P_a$ and $P_b$ are guarded at time $j$. Note that in this case, \R will remain trapped between $a$ and $b$ at time $j+1$ unless either $r_{j+1} \notin B_{r_0}(D)$, or $r_{j+1} \in P_a \cup P_b$ (in which case \C wins the game), or \C changes the strategy and stops guarding $P_a$ or $P_b$.

We now inductively define for every integer $t \leq \frac{D}{|V|} - 1$ a value $n_t \in \mathbb N$, such that one of the following two statements holds.
\begin{enumerate}[label=(\Roman*)]
    \item \label{itm:copswin} \C has won the game before time $n_t$, or has a strategy to win starting from the position at time $n_t$.
    \item \label{itm:induct} There are $a_t,b_t\in \mathbb N$ such that
    \begin{enumerate}
        \item $1 \leq b_t-a_t \leq 1 + 2^{-(t+1)} \cdot l$, and
        \item \R is trapped between $a_t$ and $b_t$ at time $n_t \leq j \leq n_{t+1}$.
    \end{enumerate}
\end{enumerate}

Essentially, this is achieved using Lemma \ref{lem:guardpath} to inductively guard $P_y$ (for some appropriate $y$), thus trapping \R, see Figure \ref{fig:generalsituation}. We point out that the values $n_t$ are not determined a priori, but depend on how the game evolves. In particular, different strategies of \R may lead to different values for $n_t$ on the same graph. Throughout the induction, we will also show that $d(r_0,r_{n_t}) \leq (t+1)\cdot |V|$, in order to make sure that $r_{n_t} \in B_{r_0}(D)$.

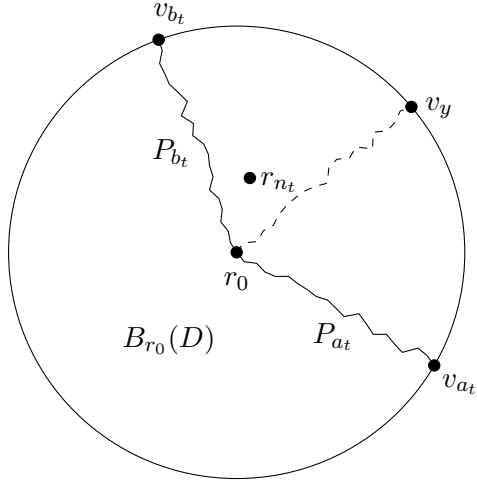
\begin{figure}
    \centering
    \begin{tikzpicture}[vertex/.style={inner sep=1.5pt,circle,draw,fill}]
	\pgfdeclarelayer{lines}
	\pgfdeclarelayer{dots}
	\pgfdeclarelayer{labels}
	\pgfsetlayers{lines,dots,labels}
	\pgfmathsetseed{1234567}
	
	\begin{pgfonlayer}{dots}
	    \node[vertex] (r0) at (0,0){};
	    \node[vertex] (van) at (-30:3){};
	    \foreach \i in {1,...,19} 
	        \node[coordinate] (va\i) at ( $(-30:3/20*\i) + rand*(60:.1)$){};
	    \node[vertex] (vbn) at (110:3){};
	    \foreach \i in {1,...,19} 
	        \node[coordinate] (vb\i) at ( $(110:3/20*\i) + rand*(200:.1)$){};
	    \node[vertex] (vy) at (40:3){};
	    \foreach \i in {1,...,19} 
	        \node[coordinate] (vy\i) at ( $(40:3/20*\i) + rand*(130:.1)$){};
	    \node[vertex] (rn) at (80:1){};
	\end{pgfonlayer}
	
	\begin{pgfonlayer}{lines}
	    \draw circle (3);
	    \draw {(0,0) \foreach \i in {1,...,19} {-- (va\i)} -- (van)};
	    \draw (0,0) \foreach \i in {1,...,19} {-- (vb\i)} -- (vbn);
	    \draw[dashed](0,0) \foreach \i in {1,...,19} {-- (vy\i)}-- (vy);
	\end{pgfonlayer}

	\begin{pgfonlayer}{labels}
	    \node [label={[label distance=2.8cm]100:$v_{b_t}$}] at (r0) {};
	    \node [label={[label distance=2.8cm]-30:$v_{a_t}$}] at (r0) {};
	    \node [label={[label distance=2.7cm]35:$v_{y}$}] at (r0) {};
	    \node [label={[label distance=1cm]-43:$P_{a_t}$}] at (r0) {};
	    \node [label={[label distance=.95cm]115:$P_{b_t}$}] at (r0) {};
	    \node [label={[label distance=.55cm]75:$r_{n_t}$}] at (r0) {};
	    \node [label={[label distance=0cm]270:$r_0$}] at (r0) {};
	    \node [label={[label distance=.7cm]260:$B_{r_0}(D)$}] at (r0) {};
	\end{pgfonlayer}
	
    \end{tikzpicture}
    \caption{Situation at time $j$ for $n_t \leq j \leq n_{t+1}$: $P_{a_t}$ and $P_{b_t}$ are guarded and \R is trapped between $a_t$ and $b_t$. By guarding a shortest $r_0$--$v_y$-path $P_y$~(dashed) we make sure that \R is trapped either between $a_t$ and $y$, or between $y$ and $b_t$.}
    \label{fig:generalsituation}
\end{figure}

To start the inductive construction, let $y = \lfloor \frac {l}{2} \rfloor$. By Lemma \ref{lem:guardpath} there is a strategy for \C to make sure that $P_0$ is $1$-guarded at all times $j>m$ and $d(r_0,r_m) \leq |V|$ for some appropriate $m$. Analogously there is a strategy to make sure that $P_y$ is $2$-guarded at all times $j>m'$ and $d(r_0,r_{m'}) \leq |V|$ for some appropriate $m'$. Since those two strategies don't interfere with each other, we have a strategy ensuring that both $P_0$ and $P_y$ are guarded for $j \geq n_0$ and $d(r_0,r_{n_0}) \leq |V|$, where $n_0 := \max(m,m')$.

Let $H$ be the graph obtained from $\hat G$ by removing $P_{0}$ and $P_{y}$. If $r_{n_0} \notin H$, then $r_{n_0} \in P_0 \cup P_y$ and \C has won the game already. If $r_{n_0}$ is in a finite component of $H$, then \C has a winning strategy by Lemma \ref{lem:planarwin}. In both of these cases \ref{itm:copswin} holds. Thus we can assume that $r_{n_0}$ lies in an infinite component $C$ of $H$. Since $r_{n_0} \in B_{r_0}(|V|)$, there must be a path in $C \cap B_{r_0}(|V|)$ connecting $r_{n_0}$ to either $[0,y]$ or $[y,l]$ and thus at time $n_0$ \R is trapped either between $0$ and $y$ or between $y$ and $l$. In the first case choose $a_0 = 0$ and $b_0 = y$, in the second case choose $a_0 = y$ and $b_0 =l$. In both cases it is straightforward to check that \ref{itm:induct} holds. 

For the induction step assume that we have defined $n_{t-1}$ as claimed. If \ref{itm:copswin} holds, then we can define $n_{t} = n_{t-1}$ and \ref{itm:copswin} still holds. So let us assume that \ref{itm:induct} holds for $n_{t-1}$. Let $y = \lfloor \frac {a_{t-1}+b_{t-1}}2 \rfloor$, let $i \in \{1,2,3\}$ be such that neither $P_{a_{t-1}}$ nor $P_{b_{t-1}}$ is $i$-guarded. Lemma \ref{lem:guardpath} provides us with a strategy such that for an appropriate $n_t$ we have that $P_y$ is $i$-guarded at all times $j \geq n_t$, and $d(r_0,r_{n_t}) \leq d(r_0,r_{n_t}) + |V| \leq (t+1) \cdot |V|$.

Let $H$ be the graph obtained from $\hat G$ by removing $P_{a_{t-1}}$, $P_{b_{t-1}}$ and $P_y$. As before, if $r_{n_t} \notin H$, then \C has won the game and \ref{itm:copswin} holds. If $r_{n_t}$ is contained in a finite component of $H$, then removing two of the three paths from $G$ already leaves it in a finite component (because $\hat G$ is planar and $P_{a_{t-1}}$, $P_{b_{t-1}}$ and $P_y$ pairwise don't cross in the embedding). Consequently, \C has a winning strategy in this situation by Lemma \ref{lem:planarwin}. Finally assume that $r_{n_t}$ is contained in an infinite component $C$ of $H$. For $n_{t-1} \leq j \leq n_t$ the paths $P_{a_{t-1}}$ and $P_{b_{t-1}}$ are guarded at time $j$ and $r_j \in B_{r_0}(D)$. Together with the assumption that \R was trapped between $a_{t-1}$ and $b_{t-1}$ at time $n_{t-1}$, this implies that \R is trapped between $a_{t-1}$ and $b_{t-1}$ at time $n_t$ unless \C has won the game before time $n_t$. Since $P_y$ is also guarded, the same argument as above gives that at time $n_t$ \R is either trapped between $a_{t-1}$ and $y$, or between $y$ and $b_{t-1}$. In the first case take $a_t = a_{t-1}$ and $b_t = y$, in the second case take $a_t = y$ and $b_t=b_{t-1}$. In both cases it is not hard to verify that \ref{itm:induct} is satisfied.

To conclude the proof, we remark that the \ref{itm:induct} can't possibly be satisfied for $t = \frac D{|V|}-1$. Indeed, in this case
\[
2^{-(t+1)} \cdot l < 2^{-\frac D{|V|}} \cdot |B_{r_0}(D)| = 2^{-\frac D{|V|} + \log (|B_{r_0}(D)|)} < 1.
\]
Since $b_t - a_t$ is an integer, it follows that $b_t - a_t = 1$, and thus $[a_t,b_t] = \emptyset$. In particular $r_{n_t}$ cannot lie between $a_t$ and $b_t$. Hence there is some $t \leq \frac D{|V|}-1$, such that \ref{itm:copswin} holds, thus \C has a winning strategy.
\end{proof}

As mentioned in the introduction, Theorem~\ref{thm:main} can be used to make progress on Conjecture~\ref{con:schroeder}. In particular, we have the following result.

\begin{corollary}
\label{cor:schroedercon}
If $G$ is a finite graph of genus $g \leq 3$, then $c(G) \leq g + 3$.
\end{corollary}

We remark that the cases $g\leq 2$ were previously known, see \cite{aignerfromme-planar,schroeder-genus}, so our only real contribution to Corollary~\ref{cor:schroedercon} is the case $g=3$. We still prove all cases for convenience. We say that a strategy of \C \emph{reduces the genus by $r$ using $s$ cops}, if it yields $i$-guarded subgraphs $H_i$ of $G$ for $1 \leq i \leq s$ such that the genus of the graph obtained from $G$ by removing all $H_i$ is at most $g-s$, where $g$ is the genus of $G$. Using this notation, we have the following result.

\begin{lemma}
\label{lem:genusreduction}
Assume that we play \CnR with $k \geq 4$ cops. Then
\begin{enumerate}
    \item \C has a strategy reducing the genus by 1 using 2 cops, and
    \item \C has a strategy reducing the genus either by 1 using 1 cop, or by 2 using 3 cops.
\end{enumerate}
\end{lemma}
\begin{proof}
The first part is implicit in \cite{quilliot-genus}, the second part is Proposition 3.2 in \cite{schroeder-genus}.
\end{proof}

\begin{proof}[Proof of Corollary \ref{cor:schroedercon}]
For $g=0$ and $g=1$ this follows directly from Theorem \ref{thm:main} (note that any planar graph can be embedded in the torus). For $g=2$ apply the first part of Lemma \ref{lem:genusreduction}, then apply Theorem \ref{thm:main} to the resulting toroidal graph. For $g=3$, first apply the second part of Lemma \ref{lem:genusreduction}. If the strategy used 1 cop to reduce the genus by 1, then apply Corollary \ref{cor:schroedercon} for $g=2$ to the remaining graph, otherwise apply Theorem~\ref{thm:main}.
\end{proof}

\section*{Acknowledgements}
The author thanks J\'er\'emie Turcotte and the anonymous referees for pointing out a flaw in the proof of Theorem~\ref{thm:main} in an earlier version of this paper and suggesting various valuable improvements.

\bibliographystyle{abbrv}
\bibliography{bibliography}

\end{document}